\documentclass[11pt]{amsart}

\calclayout

\usepackage[foot]{amsaddr}
\usepackage{graphicx}
\usepackage{color}
\usepackage{amsfonts}
\usepackage{amssymb}
\usepackage{amsbsy}
\usepackage{amsmath}
\usepackage{amsthm}
\usepackage{url}
\usepackage{enumerate}
\usepackage{subcaption} 
\usepackage{pgfplots}

\usepackage{pgfkeys}
\newenvironment{customlegend}[1][]{%
	\begingroup
	\csname pgfplots@init@cleared@structures\endcsname
	\pgfplotsset{#1}%
}{%
	\csname pgfplots@createlegend\endcsname
	\endgroup
}%
\def\addlegendimage{\csname pgfplots@addlegendimage\endcsname}

\newcommand{\R}{\mathbb{R}}
\newcommand{\N}{\mathbb{N}}

\newcommand{\J}{\mathcal{J}}
\newcommand{\Jhat}{\hat{\mathcal{J}}}
\newcommand{\Jnoncor}{\hat{J}}
\newcommand{\cJhatn}{{{\Jhat_\red}}}

\newcommand{\pr}{\textnormal{pr}}
\newcommand{\du}{\textnormal{du}}

\newcommand{\Params}{\mathcal{P}}
\newcommand{\Proj}{\mathrm{P}}

\newcommand{\bformd}{a_{\mu}}
\newcommand{\lformd}{l_{\mu}}
\newcommand{\kformd}{k_{\mu}}
\newcommand{\jformd}{j_{\mu}}
\newcommand{\resd}{r_{\mu}}
\newcommand{\cont}[1]{\gamma_{#1}}
\newcommand{\red}{{r}}

\newtheorem{theorem}{Theorem}[section]

\newtheorem{proposition}[theorem]{Proposition}
\theoremstyle{definition}

\begin{document}
\title{Model Reduction for Large Scale Systems}
\thanks{The authors acknowledge funding by the Deutsche Forschungsgemeinschaft for the project {\em Localized Reduced Basis Methods for PDE-constrained Parameter Optimization}
under contract OH 98/11-1 and by the Deutsche Forschungsgemeinschaft under Germany’s Excellence Strategy EXC 2044 390685587, Mathematics M\"unster: Dynamics -- Geometry -- Structure.}

\author[T.~Keil]{Tim Keil}

\author[M.~Ohlberger]{Mario Ohlberger}

\address[]{Mathematics M\"unster, Westf\"alische Wilhelms-Universit\"at M\"unster, Einsteinstr.~62, D-48149 M\"unster, Germany. \url{https://www.wwu.de/AMM/ohlberger}}
\email{{\tt \{tim.keil,mario.ohlberger\}@uni-muenster.de}}

\begin{abstract}
Projection based model order reduction has become a mature technique for simulation of large classes of parameterized systems. However, several challenges remain for problems where the solution manifold of the parameterized system cannot be well approximated by linear subspaces. While the online efficiency of these model reduction methods is very convincing for problems with a rapid decay of the Kolmogorov n-width, there are still major drawbacks and limitations. Most importantly, the construction of the reduced system in the offline phase is extremely CPU-time and memory consuming for large scale and multi scale systems. 
For practical applications, it is thus necessary to derive model reduction techniques that do not rely on a classical offline/online splitting but allow for more flexibility in the usage of computational resources. A promising approach with this respect is model reduction with adaptive enrichment. In this contribution we investigate Petrov-Galerkin based model reduction with adaptive basis enrichment within a Trust Region approach for the solution of multi scale and large scale PDE constrained parameter optimization.

\keywords{PDE constraint optimization  \and reduced basis method \and trust region method.}
\end{abstract}
\maketitle              
\section{Introduction}
Model order reduction (MOR) is a very active research field that has seen 
tremendous development in recent years, both from a theoretical and application  
point of view. For an introduction and overview on recent development
we refer e.g. to \cite{MR3701994}. 
A particular promising model reduction approach for parameterized partial differential equations (pPDEs)
is the Reduced Basis (RB) Method that relies on the approximation of the solution manifold of pPDEs
by low dimensional linear spaces that are spanned from suitably selected particular solutions, 
called snapshots. For time-dependent problems, the POD-Greedy method \cite{Haasdonk2008277} 
defines the {\em Gold-Standard}. As RB methods rely on so called efficient offline/online splitting, they need 
to be combined with supplementary interpolation methods in case of non-affine parameter dependence 
or non-linear differential equations. The empirical interpolation method (EIM) \cite{BarraultMadayEtAl2004} 
and its various generalizations, e.g. \cite{Drohmann2012}, are key technologies with this respect. 
While RB methods are meanwhile very well established and analyzed for scalar coercive problems, 
there are still major challenges for problems with a slow convergence of the Kolmogorov N-width \cite{OhlRav16}.
Such problems in particular include pPDEs with high dimensional or even infinite dimensional parameter dependence, 
multiscale problems as well as hyperbolic or advection dominated transport problems.  
Particular promising approaches for high dimensional parameter dependence and large or multiscale problems 
are localized model reduction approaches. We refer to \cite{buhr2019localized} for a recent review of such approaches, including the localized reduced basis multiscale method (LRBMS) \cite{Ohlberger2015A2865}. 
Several of these approaches have already been applied in multiscale applications, in particular for 
battery simulation with resolved electrode geometry and Buttler-Volmer kinetics \cite{Feinauer2019}.
Based on efficient localized a posteriori error control and online enrichment, these methods overcome 
traditional offline/online splitting and are thus particularly well suited for applications in optimization or inverse 
problems as recently demonstrated in \cite{MR3662229,Ohlberger2018143}. 
In the context of PDE constrained optimization, a promising Trust Region (TR) -- RB approach that updates the reduced model 
during the trust region iteration has recently been studied in \cite{MR3716566,keil2020nonconforming,banholzer2020adaptive}.
In the latter two contributions a new non-conforming dual (NCD) approach has been introduced that improves 
the convergence of the adaptive TR-RB algorithm in the case of different reduced spaces for the corresponding primal and 
dual equations of the first order optimality system.

While these contributions were all based on Galerkin projection of the 
respective equations, we will introduce a new approach based on Petrov-Galerkin (PG) projection in the following. 
As we will demonstrate in Section \ref{sec:PG} below, the PG-reduced optimality system is a conforming approximation which 
allows for more straight forward computation of derivative information and respective a posteriori error estimates. 
In Section \ref{sec:exp} we evaluate and compare the Galerkin and Petrov-Galerkin approaches with respect to the error behavior and the resulting TR-RB approaches with adaptive enrichment.
Although the convergence of the TR-RB method can be observed for both approaches, the results demonstrate  that the PG approach may require further investigation with respect to stabilization.

\section{Petrov-Galerkin based model reduction for PDE constrained optimization}\label{sec:PG}

\label{sec:problem}

In this contribution we consider the following class of PDE constrained minimization problems:

{\color{white}
	\begin{equation}
	\tag{P}
	\label{P}
	\end{equation}
}\vspace{-40 pt} 
\begin{subequations}\begin{align}
	& \min_{\mu \in \Params} \J(u_\mu, \mu),
	&&\text{with } \J(u, \mu) = \Theta(\mu) + j_\mu(u) + k_\mu(u, u),
	\tag{P.a}\label{P.argmin}\intertext{%
		subject to $u_\mu \in V$ being the solution of the \emph{state -- or primal -- equation}
	}
	&a_\mu(u_\mu, v) = l_\mu(v) &&\text{for all } v \in V,
	\tag{P.b}\label{P.state}
	\end{align}\end{subequations}%

\setcounter{equation}{0}
where $\Theta \in \Params \to \mathbb{R}$ denotes a parameter functional. 

Here, $V$ denotes a real-valued Hilbert space with inner product $(\cdot \,,\cdot)$ and its induced norm $\|\cdot\|$, 
$\Params \subset \R^P$, with $P \in \N$ denotes a compact and convex admissible parameter set and 
 $\J: V \times \Params \to \R$ a quadratic continuous functional.
 In particular, we consider box-constraints of the form 
\[ 
\Params:= \left\{\mu\in\mathbb{R}^P\,|\,\mu_\mathsf{a} \leq \mu \leq \mu_\mathsf{b} \right\} \subset \R^P,
\] 
for given parameter bounds $\mu_\mathsf{a},\mu_\mathsf{b}\in\mathbb{R}^P$, where {``$\leq$''} has to be understood component-wise. 

For each admissible parameter $\mu \in \Params$, $a_\mu: V \times V \to \R$ denotes a continuous and coercive bilinear form,
$l_\mu, j_\mu: V \to \R$ are continuous linear functionals and $k_\mu: V \times V \to \R$ denotes a continuous symmetric bilinear form. 
The primal residual of \eqref{P.state} is key for the optimization as well as for a posteriori error estimation. 
We define for given $u \in V$, $\mu \in \Params$, the primal residual $r_\mu^\pr(u) \in V'$ associated with \eqref{P.state} by
\begin{align}
r_\mu^\pr(u)[v] := l_\mu(v) - a_\mu(u, v) &&\text{for all }v \in V.
\label{eq:primal_residual}
\end{align}

Following the approach of \emph{first-optimize-then-discretize}, we base our 
discretization and model order reduction approach on the first order necessary optimality system, i.e. (cf. \cite{banholzer2020adaptive} for details and further references)
\begin{subequations}
		\label{eq:optimality_conditions}
		\begin{align}
		r_{\bar \mu}^\pr(\bar u)[v] &= 0 &&\text{for all } v \in V,
		\label{eq:optimality_conditions:u}\\
		\partial_u \J(\bar u,\bar \mu)[v] - a_\mu(v,\bar p) &= 0 &&\text{for all } v \in V,
		\label{eq:optimality_conditions:p}\\
		(\partial_\mu \J(\bar u,\bar \mu)+\nabla_{\mu} r^\pr_{\bar\mu}(\bar u)[\bar p]) \cdot (\nu-\bar \mu) &\geq 0 &&\text{for all } \nu \in \Params. 
		\label{eq:optimality_conditions:mu}
		\end{align}	
	\end{subequations}
From \eqref{eq:optimality_conditions:p} we deduce the so-called \emph{adjoint -- or dual -- equation}  
\begin{align}
a_\mu(q, p_\mu) = \partial_u \J(u_\mu, \mu)[q]
= j_\mu(q) + 2 k_\mu(q, u_\mu)&&\text{for all } q \in V,
\label{eq:dual_solution}
\end{align}
with solution $p_{\mu} \in V$ for a fixed $\mu \in \Params$ and given the solution $u_\mu \in V$ to the state equation \eqref{P.state}.
For given $u, p \in V$, we introduce the dual residual $r_\mu^\du(u, p) \in V'$ associated with \eqref{eq:dual_solution} as
\begin{align}
r_\mu^\du(u, p)[q] := j_\mu(q) + 2k_\mu(q, u) - a_\mu(q, p)&&\text{for all }q \in V.
\label{eq:dual_residual}
\end{align}

\subsection{Petrov-Galerkin based discretization and model reduction}
Assuming $V_h^\pr, V_h^\du \subset V$ to be a finite-dimensional subspaces, we define a Petrov-Galerkin projection of (\ref{P}) onto $V_h^\pr, V_h^\du$ by
considering, for each $\mu \in \Params$, the solution $u_{h, \mu} \in V_h^\pr$ of the \emph{discrete primal equation}
\begin{align}
\bformd(u_{h, \mu}, v_h) = \lformd(v_h) &&\text{for all } v_h \in V_h^\du,
\label{eq:state_h}
\end{align}
and for given $\mu, u_{h, \mu}$, the solution $p_{h, \mu} \in V_h^\du$ of the \emph{discrete dual equation} as
\begin{align}
\bformd(q_h, p_{h, \mu}) = \partial_u \J(u_{h, \mu}, \mu)[q_h] = \jformd(q_h) + 2 \kformd(q_h, u_{h, \mu}) &&\forall q_h \in V_h^\pr.
\label{eq:dual_solution_h}
\end{align}
Note that the test space of one equation corresponds with the ansatz space of the other equation. In order to obtain a quadratic system, we require 
 $\dim V_h^\pr = \dim V_h^\du$. Note that a Ritz-Galerkin projection is obtained, if the primal and dual discrete spaces coincide, i.e. $V_h^\pr = V_h^\du$. 

Given problem adapted reduced basis (RB) spaces $V_\red^\pr \subset V_h^\pr, V_\red^\du \subset V_h^\du$ of the same low dimension $n := \dim V_\red^\pr = \dim V_\red^\du$  we obtain the reduced versions for the optimality system as follows:
\begin{subequations}
	\label{eq:optimality_conditionsRB}
	\begin{itemize}
		\item PG-RB approximation for \eqref{eq:optimality_conditions:u}: For each $\mu \in \Params$ the primal variable $u_{\red, \mu} \in V_\red^\pr$ 
		of the \emph{RB approximate primal equation} is defined through
		\begin{align}
		\bformd(u_{\red, \mu}, v_\red) = \lformd(v_\red) &\qquad \text{for all } v_\red \in V_\red^\du.
		\label{eq:state_red}
		\end{align}
		\item PG-RB approximation for \eqref{eq:optimality_conditions:p}: For each $\mu \in \Params$, $u_{\red, \mu} \in V_\red^\pr$ the dual/adjoint variable $p_{\red, \mu} \in V_\red^\du$ satisfies the \emph{RB approximate dual equation} 
		\begin{align}
		\bformd(q_\red, p_{\red, \mu}) = \partial_u \J(u_{\red, \mu}, \mu)[q_\red] = \jformd(q_\red) + 2 \kformd(q_\red, u_{\red, \mu}) &&\forall q_\red \in V_\red^\pr.
		\label{eq:dual_solution_red}
		\end{align}
	\end{itemize}
\end{subequations}

We define the PG-RB reduced optimization functional by 
\begin{align}
\cJhatn(\mu) := \J(u_{\red, \mu}, \mu) 
\label{eq:Jhat_red}
\end{align} 
with $u_{\red, \mu} \in V_\red^\pr$ being the solution of \eqref{eq:state_red}.
We then consider the \emph{RB reduced optimization problem} by finding a locally optimal solution $\bar \mu_\red$ of
\begin{align}
\min_{\mu \in \Params} \cJhatn(\mu).
\tag{$\hat{\textnormal{P}}_\red$}\label{Phat_\red}
\end{align}
Note that in contrast to the NCD-approach that has been introduced in \cite{keil2020nonconforming,banholzer2020adaptive}, we do not need to correct the 
reduced functional in our PG-RB approach, as the primal and dual solutions automatically
satisfy $r_\mu^\pr(u_{\red,\mu})[p_{\red,\mu}] = 0$. Actually, this is the main motivation for the usage of the Petrov-Galerkin approach in this contribution. 
In particular this results in the possibility to compute the gradient of the reduced functional with respect to the parameters solely based on the primal and dual solution of the 
PG-RB approximation, i.e.
\begin{align} \label{eq:red_gradient}
	\big(\nabla_\mu \cJhatn(\mu)\big)_i & = \partial_{\mu_i}\J(u_{\red,\mu},\mu) + \partial_{\mu_i}r_\mu^\pr(u_{\red,\mu})[p_{\red,\mu}]
\end{align}
for all $1 \leq i \leq P$ and $\mu \in \Params$, where $u_{\red, \mu} \in V_\red^\pr$ and $p_{\red, \mu} \in V_\red^\du$ denote the PG-RB primal and dual reduced solutions of \eqref{eq:state_red} and \eqref{eq:dual_solution_red}, respectively.
Note that in \cite{MR3716566}, the formula in \eqref{eq:red_gradient} was motivated by replacing the full order functions by their respective reduced counterpart which resulted in an inexact gradient for the non-conforming approach.
In the PG setting, \eqref{eq:red_gradient} instead defines the true gradient of $\cJhatn$ without having to add a correction term as proposed in \cite{keil2020nonconforming,banholzer2020adaptive}.
This also holds for the true Hessian which can be computed by also replacing all full order counterparts in the full order Hessian.
In this contribution we however only focus on quasi-Newton methods that do not require a reduced Hessian. 

\subsection{A posteriori error estimation for an error aware algorithm}
\label{sec:a_posteriori_estimation}
In order to construct an adaptive trust region algorithm we require a posteriori error estimation for \eqref{eq:state_red}, \eqref{eq:dual_solution_red} and \eqref{eq:Jhat_red}.
For that, we can utilize standard residual based estimation. For an a posteriori result of the gradient $\nabla_\mu \cJhatn(\mu)$, we also refer to \cite{keil2020nonconforming}.
\begin{proposition}[Upper error bound for the reduced quantities]	\label{prop:error_reduced_quantities}
	For $\mu \in \Params$, let $u_{h, \mu} \in V_h^\pr$ and $p_{h, \mu} \in V_h^\du$ be solutions of \eqref{eq:state_h} and \eqref{eq:dual_solution_h} and let $u_{\red, \mu} \in V_\red^\pr$, $p_{\red, \mu} \in V_\red^\du$ be 
	a solution of \eqref{eq:state_red}, \eqref{eq:dual_solution_red}. Then it holds
	\begin{enumerate}[(i)]
		\item
		$\|u_{h, \mu} - u_{\red, \mu}\| \leq \Delta_\pr(\mu) := \alpha_\mu^{-1}\, \|\resd^\pr(u_{\red, \mu})\|$,
		\item 	$\|p_{h, \mu} - p_{\red, \mu}\| \leq \Delta_\du(\mu) := \alpha_\mu^{-1}\big(2 \cont{\kformd}\;\Delta_\pr(\mu) + \|\resd^\du(u_{\red, \mu}, p_{\red, \mu})\|\Big)$,
		\item  $|\Jhat_h(\mu) - \cJhatn(\mu)| \leq \Delta_{\cJhatn}(\mu)
		:=  \Delta_\pr(\mu) \|\resd^\du(u_{\red, \mu}, p_{\red,\mu})\| + \Delta_\pr(\mu)^2 \cont{\kformd}$
	\end{enumerate} 
where $\alpha_\mu$ and $\cont{\kformd}$ define the inf-sup stability constant of $\bformd$ and the continuity constant of $\kformd$, respectively.
\end{proposition}
\begin{proof}
	For $(i)$ and $(ii)$ we rely on the inf-sup stability of $\bformd$ on $V_h^\pr$ and $V_h^\du$ and proceed analogously to \cite{MR3716566}. For $(iii)$, we refer to \cite{keil2020nonconforming} since $r_\mu^\pr(u_{\red,\mu})[p_{\red,\mu}] = 0$.
\end{proof}
It is important to mention that the computation of these error estimators include the computation of the (parameter dependent) inf-sup constant of $\bformd$ which involves an eigenvalue problem on the FOM level.
In practice, cheaper techniques such as the successive constraint method \cite{huynh2015methods} can be used.
For the conforming approach $V_h^\pr = V_h^\du$, the inf-sup constant is equivalent to the coercivity constant $\underline{\bformd}$ of $\bformd$ which can be cheaply bounded
from below with the help of the min-theta approach, c.f. \cite{Ha14}.

\subsection{Trust-Region optimization approach and adaptive enrichment}
\label{sec:TRRBt}

Error aware Trust-Region - Reduced Basis methods (TR-RB) with several different advances and features have been extensively studied e.g. in \cite{MR3716566,keil2020nonconforming,banholzer2020adaptive}. 
They iteratively compute a first-order critical point of problem \eqref{P}.
For each outer iteration $k\geq 0$ of the TR method, we consider a model function $m^{(k)}$ as a cheap local approximation of the quadratic cost functional $\J$ in the so-called trust-region,
which has radius $\varrho^{(k)}$ and can be characterized by the a posteriori error estimator. 
In our approach we choose
$$m^{(k)}(\cdot):= \cJhatn^{(k)}(\mu^{(k)}+\cdot)$$ for $k\geq 0$,
where the super-index $(k)$ indicates that we use different RB spaces $V_\red^{*, (k)}$ in each iteration.
Thus, we can use $\Delta_{\cJhatn}(\mu)$ for characterizing the trust-region.
We are therefore interested in solving the following error aware constrained optimization sub-problem

\begin{equation}
\label{TRsubprob}
\begin{aligned}
\min_{\widetilde{\mu}\in\Params} \cJhatn^{(k)}(\widetilde{\mu}) \quad \text{ s.t. } \quad  &\frac{\Delta_{\Jhat}(\widetilde{\mu})}{\cJhatn^{(k)}(\widetilde{\mu})}\leq \varrho^{(k)}, \quad
\widetilde{\mu}:= \mu^{(k)}+s \in\Params \\ & \text{ and } r_{\tilde{\mu}}^\pr(u_{\tilde{\mu}})[v]= 0 \, \text{ for all }  v\in V.
\end{aligned}
\end{equation}

In our TR-RB algorithm we build on the algorithm in \cite{keil2020nonconforming}. In the sequel, we only summarize the main features of the algorithm and refer to the source for more details.
We initialize the RB spaces with the starting parameter $u_{\mu^{(0)}}$, i.e.~$V^{\pr,(0)}_\red = \big\{u_{h,\mu^{(0)}}\big\}$ and $V^{\du,(0)}_\red = \big\{p_{h,\mu^{(0)}}\big\}$.
For every iteration point $\mu_k$, we locally solve \eqref{TRsubprob} with the quasi-Newton projected BFGS algorithm combined with an Armijo-type condition and terminate with a standard reduced FOC termination criteria,
modified with a projection on the parameter space $\Proj_\Params$ to account for constraints on the parameter space.
Additionally, we use a second boundary termination criteria for preventing the subproblem from spending too much computational time on the boundary of the trust region. 
After the next iterate $\mu_{k+1}$ has been computed, the sufficient decrease conditions helps to decide whether to accept the iterate:
\begin{align}
\label{Suff_decrease_condition}
\cJhatn^{(k+1)}(\mu^{(k+1)})\leq \cJhatn^{(k)}(\mu_\text{\rm{AGC}}^{(k)}) && \text{ for all } k\in\mathbb{N}
\end{align}
This condition can be cheaply checked with the help of an sufficient and necessary condition.
If $\mu_{k+1}$ is accepted, we use the parameter to enrich the RB spaces in a Lagrangian manner, i.e.
$
V^{\pr,k}_{\red} = V^{\pr,k-1}_{\red} \cup \{u_{h,\mu}\}, 
V^{\du,k}_{\red} = V^{\du,k-1}_{\red} \cup \{p_{h,\mu}\}.
$
For more basis constructions, we refer to \cite{banholzer2020adaptive} where also a strategy is proposed to skip an enrichment.
After the enrichment, overall convergence of the algorithm can be checked with a FOM-type FOC condition
$$\|\mu^{(k+1)}-\Proj_\Params(\mu^{(k+1)}-\nabla_\mu \Jhat_h(\mu^{(k+1)}))\|_2\leq \tau_{\text{\rm{FOC}}},$$
where the FOM quantities are available from the enrichment. Moreover, these FOM quantities allow for computing a condition for possibly enlargement of the TR radius if the reduced model is better than expected.
For this, we use
\begin{equation}
\label{TR_act_decrease}
\varrho^{(k)}:= \frac{\Jhat_h(\mu^{(k)})-\Jhat_h(\mu^{(k+1)})}{\cJhatn^{(k)}(\mu^{(k)})-\cJhatn^{(k)}(\mu^{(k+1)})}  \geq  \eta_\varrho
\end{equation}
For the described algorithm the following convergence result holds.
\begin{theorem}[Convergence of the TR-RB algorithm, c.f.~\cite{banholzer2020adaptive}]
	\label{Thm:convergence_of_TR}
	For sufficient assumptions on the Armijo search to solve \eqref{TRsubprob}, every accumulation point $\bar\mu$ of the sequence $\{\mu^{(k)}\}_{k\in\mathbb{N}}\subset \Params$ generated by the above described TR-RB algorithm
	is an approximate first-order critical point for $\Jhat_h$, i.e., it holds
	\begin{equation}
	\label{First-order_critical_condition}
	\|\bar \mu-\Proj_\Params(\bar \mu-\nabla_\mu \Jhat_h(\bar \mu))\|_2 = 0.
	\end{equation}
\end{theorem}
Note that the above Theorem is not affected from the choice of the reduced primal and dual equations as long as Proposition~\ref{prop:error_reduced_quantities} holds.
In fact, the result can also be used for the TR-RB with the proposed PG approximation of these equations.

Let us also mention that it has been shown in \cite{banholzer2020adaptive} that a projected Newton algorithm for the subproblems can enhance the convergence speed and accuracy.
Moreover, a reduced Hessian can be used to introduce an a posteriori result for the optimal parameter which can be used as post processing.
In the work at hand, we neglect to transfer the ideas from \cite{banholzer2020adaptive} for the PG variant.  

\section{Numerical experiments}\label{sec:exp}

In this section, we analyze the behavior of the proposed PG variant of the TR-RB (BFGS PG TR-RB) algorithm.
We aim to compare the computational time and the accuracy of the PG variant to existing approaches from the literature.
To this end, we mention the different algorithms that we compare in this contribution: \\
\textbf{Projected BFGS FOM:} As FOM reference optimization method, we consider a standard projected BFGS method,
which uses FOM evaluations for all required quantities. \\
\textbf{Non-conforming TR-RB algorithm from \cite{keil2020nonconforming} (BFGS NCD TR-RB):} 
For comparison, we choose the above described TR-RB algorithm but with a non conforming choice of the RB spaces, which means that \eqref{eq:state_red} and \eqref{eq:dual_solution_red} are Galerkin projected equations where the test space coincides with the respective ansatz space.
This requires to use the NCD-corrected reduced objective functional $ \Jhat_\red(\mu) + r_\mu^\pr(u_{\red,\mu})[p_{\red,\mu}]$ and requires additional quantities for computing the Gradient.

For computational details including the choice of all relevant tolerances for both TR-RB algorithms we again refer to \cite{keil2020nonconforming}, where all details can be found. We only differ in the choice of the stopping tolerance $\tau_{\text{\rm{FOC}}} = 10^{-6}$.
Also note that, as pointed out in Section \ref{sec:a_posteriori_estimation}, we use the same error estimators as in \cite{keil2020nonconforming} which include the coercivity constant instead of the inf-sub constant.
The source code for the presented experiments can be found in \cite{Code}, which is a revised version of the source code for \cite{keil2020nonconforming} and also contains detailed interactive \texttt{jupyter}-notebooks\footnote{Available at \url{https://github.com/TiKeil/Petrov-Galerkin-TR-RB-for-pde-opt}.}.

\subsection{Model problem: Quadratic objective functional with elliptic PDE constraints}
\label{sec:model_problem}
For our numerical evaluation we reconsider  Experiment 1 from \cite{banholzer2020adaptive}.
We set the objective functional to be a weighted $L^2$-misfit on a domain of interest $D \subseteq \Omega$ with a weighted Tikhonov term, i.e.
\begin{align*}
\mathcal{J}(v, \mu) = \frac{\sigma_d}{2} \int_{D}^{} (v - u^{\text{d}})^2 + \frac{1}{2} \sum^{M}_{i=1} \sigma_i (\mu_i-\mu^{\text{d}}_i)^2 + 1,
\end{align*} 
with a desired state $u^{\text{d}}$ and desired parameter $\mu^{\text{d}}$. We added the constant term $1$ to verify that $\mathcal{J}>0$. 
With respect to the formulation in \eqref{P.argmin}, we have
$\Theta(\mu) = \frac{\sigma_d}{2} \sum^{M}_{i=1} \sigma_i (\mu_i-\mu^{\text{d}}_i)^2 + \frac{\sigma_d}{2} \int_{D}^{} u^{\text{d}} u^{\text{d}}$, $
j_{\mu}(u) = -\sigma_d \int_{D}^{} u^{\text{d}}u$, and $k_{\mu}(u,u) = \frac{\sigma_d}{2} \int_{D}^{} u^2$.
From this general choice we can construct several applications. One instance is to consider the stationary heat equation as equally constraints, i.e. the weak formulation of the parameterized equation
\begin{equation} \label{eq:prot_state}
\begin{split}
-  \nabla \cdot \left( \kappa_{\mu}  \nabla u_{\mu} \right) &= f_{\mu} \hspace{58pt} \text{in } \Omega, \\
c_\mu ( \kappa_{\mu}  \nabla u_{\mu} \cdot n) &= (u_{\text{out}} - u_{\mu}) \hspace{20pt} \text{on } \partial \Omega.
\end{split}
\end{equation}
with parametric diffusion coefficient $\kappa_{\mu}$ and source $f_{\mu}$, outside temperature $u_{\text{out}}$ and Robin function $c_\mu$. 
From this the bilinear and linear forms $a_{\mu}$ and $l_\mu$ can easily be deduced
and we set the parameter box constraints
$
\mu_i \in [\mu_i^{\text{min}}, \mu_i^{\text{max}}].
$

As an application of \eqref{eq:prot_state}, we use the blueprint of a building with windows, heaters, doors and walls, which can be parameterized according to Figure \ref{ex1:blueprint}.
We picked a certain domain of interest $D$ and we enumerated all windows, walls, doors and heaters separately.
\begin{figure}
	\begin{subfigure}[b]{0.60\textwidth}
		\includegraphics[width=\textwidth]{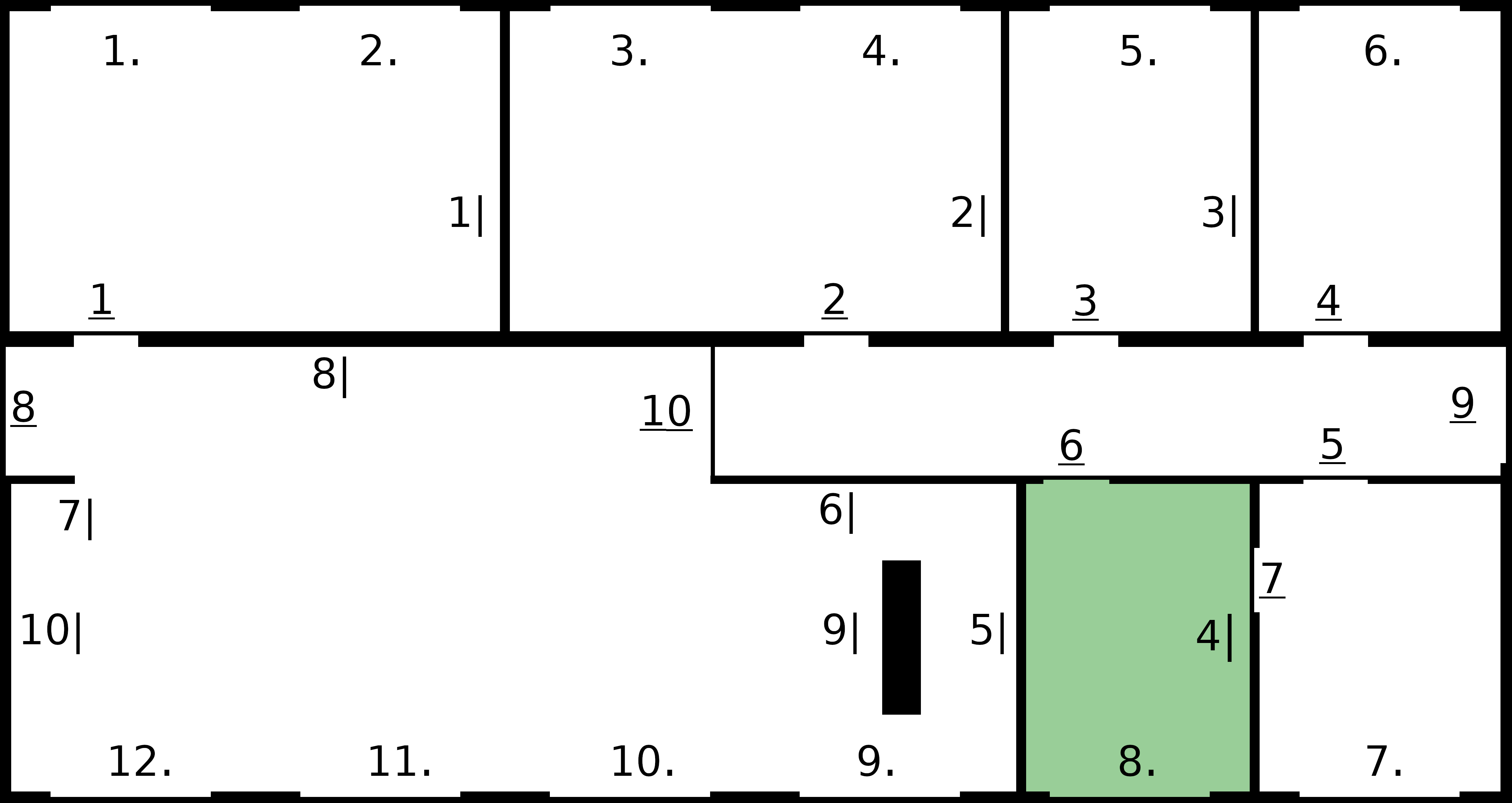}
	\end{subfigure}
	\centering
	\captionsetup{width=\textwidth}
	\caption{\footnotesize{%
			Parameterization based on a blueprint of a building floor, see \cite{keil2020nonconforming} for details. Numbers indicate potential parameters, where $i.$ is a window or a heater, $\underbar{j}$ are doors, and $k|$ are walls. The green region illustrates the domain of interest $D$.
	}}
	\label{ex1:blueprint}
\end{figure}
We set the computational domain to $\Omega := [0,2] \times [0,1] \subset \mathbb{R}^2$ and we model boundary conditions by incorporating all walls and windows that touch the boundary of the blueprint to the Robin function $c_\mu$. 
All other diffusion components enter the diffusion coefficient $\kappa_\mu$, whereas the heaters are incorporated as a source term on the right hand side $f_\mu$.
Moreover, we assume an outside temperature of $u_{\text{out}}=5$.
For our discretization we choose the conforming case for the FOM, i.e. $V_h^\pr = V_h^\du$, and use a mesh size $h= \sqrt{2}/200$ which resolves all features from the given picture and results in $\dim V_h^\pr = 80601$ degrees of freedom.

For the parameterization, we choose 2 doors, 7 heaters and 3 walls which results in a parameter space of 12 dimensions. More details can be viewed in the accompanying source code.  

\subsection{Analysis of the error behavior}
\label{sec:error_behavior}
This section aims to show and discuss the model reduction error for the proposed Petrov--Galerkin approach.
Furthermore, we compare it to the Galerkin strategy from the NCD-corrected approach, where we denote the solutions by $u_{\red}^G$, $p_{\red}^G$ and the corrected functional by $\cJhatn^\text{NCD}$, as well as to the non corrected approach from \cite{MR3716566} whose functional we denote by $\hat{J}_\red$. 
For this, we employ a standard goal oriented Greedy search algorithm as also done in \cite[Section 4.3.1]{keil2020nonconforming}
with the relative a posteriori error of the objective functional $\Delta_{\cJhatn}(\mu)/\cJhatn(\mu)$.
As pointed out in Section \ref{sec:a_posteriori_estimation}, due to $V_h^\pr = V_h^\du$, we can replace the inf-sup constant by a lower bound  for the coercivity constant of the conforming approach.
It is also important to mention that for this experiment, we have simplified our objective functional $\J$ by setting the domain of interest to the whole domain $D \equiv \Omega$.
As a result the dual problem is simpler which enhances the stability of the PG approach. The reason for that is further discussed below.
Figure \ref{fig:estimator_study} shows the difference in the decay and accuracy of the different approaches.
It can clearly be seen that the NCD-corrected approach remains the most accurate approach while the objective functional and the gradient of the PG approach shows a better approximation compared to the non corrected version.
Clearly the PG approximation of the primal and dual solutions are less accurate and at least the primal error  decays sufficiently. 

\begin{figure}[t]
	\centering%
	\footnotesize%
	\begin{tikzpicture}
	
	\definecolor{color0}{rgb}{0.65,0,0.15}
	\definecolor{color1}{rgb}{0.84,0.19,0.15}
	\definecolor{color2}{rgb}{0.96,0.43,0.26}
	\definecolor{color3}{rgb}{0.99,0.68,0.38}
	\definecolor{color4}{rgb}{1,0.88,0.56}
	\definecolor{color5}{rgb}{0.67,0.85,0.91}
	\definecolor{color6}{rgb}{0.27,0.46,0.71}
	\definecolor{color7}{rgb}{0.19,0.21,0.58}
	
	\begin{axis}[
	name=top_middle,
	width=6cm,
	height=4.5cm,
	xshift=2cm,
	log basis y={10},
	tick align=outside,
	ytick pos=right,
	xtick pos=bottom,
	x grid style={white!69.0196078431373!black},
	xmajorgrids,
	xmin=2, xmax=62,
	xtick style={color=black},
	y grid style={white!69.0196078431373!black},
	ymajorgrids,
	ymin=1e-05, ymax=1e03,
	ymode=log,
	ytick style={color=black}
	]
	\addplot [semithick, color6, mark=square*, mark size=2, mark options={solid, rotate=45, fill opacity=0.5}]
	table {%
            4    43.0423546 
            8     4.7622899 
           12     3.8108226 
           16     2.3004248 
           20     0.9208100 
           24     0.1874082 
           28     0.1138518 
           32     0.0583603 
           36     0.0300142 
           40     0.0269577 
           44     0.0338155 
           48     0.0187490 
           52     0.0172461 
           56     0.0165199 
           60     0.0077635 
 };
 \addplot [semithick, color0, mark=diamond*, mark size=2, mark options={solid, rotate=180, fill opacity=0.5}]
 table {%
            4   4.3567802 
            8   0.9740319 
           12   0.7101863 
           16   0.2501180 
           20   0.0288619 
           24   0.0110932 
           28   0.0053377 
           32   0.0025348 
           36   0.0010780 
           40   0.0002019 
           44   0.0001357 
           48   0.0000615 
           52   0.0000577 
           56   0.0000448 
           60   0.0000329 
 };
 \addplot [semithick, color3, mark=triangle*, mark size=2, mark options={solid, fill opacity=0.5}]
 table {%
            4  3.9085827 
            8  1.2450593 
           12  0.7329396 
           16  0.3251995 
           20  0.0247842 
           24  0.0110414 
           28  0.0072402 
           32  0.0094383 
           36  0.0029239 
           40  0.0019344 
           44  0.0011918 
           48  0.0000691 
           52  0.0000788 
           56  0.0000418 
           60  0.0000387 
 };
 \addplot [semithick, color6, mark=square*, mark size=2, mark options={solid, rotate=180, fill opacity=0.5}]
 table {%
            4  113.8903117 
            8   78.8674009 
           12   66.5765970 
           16   41.4825594 
           20    3.3217361 
           24    0.9909078 
           28    0.6019228 
           32    0.5208526 
           36    0.3999896 
           40    0.0916279 
           44    0.0960969 
           48    0.0418367 
           52    0.0430256 
           56    0.0365259 
           60    0.0224140 
 };
 \addplot [semithick, color0, mark=diamond*, mark size=2, mark options={solid, rotate=90, fill opacity=0.5}]
 table {%
            4  116.8902273 
            8   71.2875380 
           12   53.3466076 
           16   21.9862603 
           20    1.4742002 
           24    0.4527781 
           28    0.2880695 
           32    0.1328267 
           36    0.0677626 
           40    0.0183905 
           44    0.0082919 
           48    0.0036965 
           52    0.0053584 
           56    0.0018361 
           60    0.0018599 
 };
 \addplot [semithick, color3, mark=triangle*, mark size=2, mark options={solid, rotate=180, fill opacity=0.5}]
 table {%
            4 116.9413910 
            8  81.3719586 
           12  43.7234391 
           16  37.3951660 
           20   1.3646116 
           24   0.6264351 
           28   0.3770512 
           32   0.4323930 
           36   0.1604367 
           40   0.0830978 
           44   0.1297149 
           48   0.0083390 
           52   0.0044400 
           56   0.0021994 
           60   0.0033864 
 };

 \legend{};
 \end{axis}
 
 \begin{axis}[
 name=top_right,
 at=(top_middle.east),
 anchor=west,
 xshift=1.2cm,
 width=6cm,
 height=4.5cm,
 legend cell align={left},
 legend style={fill opacity=0.8, draw opacity=1, text opacity=1, at={(1.2,0)}, anchor=south, draw=white!80!black},
 log basis y={10},
 tick align=outside,
 x grid style={white!69.0196078431373!black},
 xmajorgrids,
 xmin=2, xmax=62,
 xtick style={color=black},
 xtick pos=bottom,
 y grid style={white!69.0196078431373!black},
 ymajorgrids,
 ymin=1e-05, ymax=1e03,
 ymode=log,
 yticklabels={,,},
 ytick pos=left,
 ytick style={color=black}
 ]
 \addplot [semithick, color0, mark=*, mark size=2, mark options={solid, fill opacity=0.5}]
 table {%
            4   0.6979035 
            8   0.1683773 
           12   0.0880920 
           16   0.0661280 
           20   0.0349068 
           24   0.0068316 
           28   0.0045575 
           32   0.0036580 
           36   0.0026852 
           40   0.0012673 
           44   0.0009064 
           48   0.0005655 
           52   0.0003778 
           56   0.0003298 
           60   0.0002402 
   };
   \addplot [semithick, color3, mark=triangle*, mark size=2, mark options={solid, rotate=180, fill opacity=0.5}]
   table {%
            4 0.7062068 
            8 0.3404906 
           12 0.1114368 
           16 0.1809833 
           20 0.1392032 
           24 0.0474685 
           28 0.0073499 
           32 0.0077065 
           36 0.0051511 
           40 0.0101835 
           44 0.0090959 
           48 0.0016053 
           52 0.0006615 
           56 0.0005745 
           60 0.0006952 
 };
 \addplot [semithick, color0, mark=pentagon*, mark size=2, mark options={solid, rotate=90, fill opacity=0.5}]
 table {%
            4    40.7460763 
            8    38.8899377 
           12    33.2418317 
           16    23.2686215 
           20     9.4626144 
           24     4.1539872 
           28     3.4412535 
           32     2.8128896 
           36     2.3952862 
           40     1.2969305 
           44     1.1588881 
           48     0.9383411 
           52     0.7765199 
           56     0.6220427 
           60     0.4001464 
 };
 \addplot [semithick, color3, mark=star, mark size=2, mark options={solid, fill opacity=0.5}]
 table {%
            4   40.8146894 
            8   39.3312537 
           12   35.1527006 
           16   35.3298773 
           20   10.6842491 
           24   10.9024490 
           28    8.1679185 
           32   12.9165319 
           36   11.0609593 
           40    8.8701537 
           44   43.5757322 
           48    5.3983823 
           52    7.8105103 
           56    5.2246835 
           60   10.7762531 
	};
	\legend{};
	\end{axis}
	
	\node[anchor=north, yshift=-15pt] at (top_middle.south) {greedy extension step};
	\node[anchor=north, yshift=-15pt] at (top_right.south) {greedy extension step};
	\node[anchor=south, yshift=4pt, xshift=2.2cm] at (top_middle.north west) {(A) functional $\Jhat_h$ and gradient $\nabla \Jhat_h$};
	\node[anchor=south, yshift=4pt, xshift=2.1cm] at (top_right.north west) {(B) primal and dual solution};
	
	\end{tikzpicture}
	\begin{tikzpicture}
	
	\definecolor{color0}{rgb}{0.65,0,0.15}
	\definecolor{color1}{rgb}{0.84,0.19,0.15}
	\definecolor{color2}{rgb}{0.96,0.43,0.26}
	\definecolor{color3}{rgb}{0.99,0.68,0.38}
	\definecolor{color4}{rgb}{1,0.88,0.56}
	\definecolor{color5}{rgb}{0.67,0.85,0.91}
	\definecolor{color6}{rgb}{0.27,0.46,0.71}
	\definecolor{color7}{rgb}{0.19,0.21,0.58}
	
	\begin{customlegend}[legend cell align={left}, legend style={fill opacity=0.8, draw opacity=1, text opacity=1,
		at=(top_right.south),
		anchor=north,
		xshift=-5.5cm,
		yshift=-1.2cm,
		inner sep=5pt,
		/tikz/column 2/.style={
			column sep=4pt},
		/tikz/column 4/.style={
			column sep=4pt},
		/tikz/column 6/.style={
			column sep=4pt},
		/tikz/column 8/.style={
			column sep=4pt},
		draw=white!80!black}, 
	legend columns=4,
	legend entries={
		$|\Jhat_h - \Jnoncor_\red|$,
		$|\nabla\Jhat_h - \tilde{\nabla}\Jnoncor_\red|$,
		$|u_{h,\mu} - u^{G}_{\red,\mu}|$,
		$|p_{h,\mu} - p^{G}_{\red,\mu}|$, 
		$|\Jhat_h - \cJhatn^{\text{NCD}}|$,
		$|\nabla\Jhat_h - \nabla\cJhatn^{\text{NCD}}|$,
		$|u_{h,\mu} - u_{\red,\mu}|$,
		$|p_{h,\mu} - p_{\red,\mu}|$,
		$|\Jhat_h - \cJhatn|$,
		$|\nabla\Jhat_h - \nabla\cJhatn|$
	}]
	\addlegendimage{semithick, color6, mark=square*, mark size=2, mark options={solid, rotate=45, fill opacity=0.5}}
	\addlegendimage{semithick, color6, mark=square*, mark size=2, mark options={solid, rotate=180, fill opacity=0.5}}
	\addlegendimage{semithick, color0, mark=*, mark size=2, mark options={solid, fill opacity=0.5}}
	\addlegendimage{semithick, color0, mark=pentagon*, mark size=2, mark options={solid, rotate=90, fill opacity=0.5}}
	\addlegendimage{semithick, color0, mark=diamond*, mark size=2, mark options={solid, rotate=180, fill opacity=0.5}}
	\addlegendimage{semithick, color0, mark=diamond*, mark size=2, mark options={solid, rotate=90, fill opacity=0.5}}
	\addlegendimage{semithick, color3, mark=triangle*, mark size=2, mark options={solid, rotate=180, fill opacity=0.5}}
	\addlegendimage{semithick, color3, mark=star, mark size=2, mark options={solid, fill opacity=0.5}}
	\addlegendimage{semithick, color3, mark=triangle*, mark size=2, mark options={solid, fill opacity=0.5}}
	\addlegendimage{semithick, color3, mark=triangle*, mark size=2, mark options={solid, rotate=180, fill opacity=0.5}}	
	\end{customlegend}
	\end{tikzpicture}
	
	\caption{%
		Evolution of the true reduction error in the reduced functional and gradient and its approximations (A)
		and the primal and dual solutions and its approximations (B), during adaptive greedy basis generation.
		Depicted is the $L^\infty(\Params_\textnormal{val})$-error for a validation set $\Params_\textnormal{val} \subset \Params$ of $100$ randomly selected parameters, i.e.~$|\Jhat_h - \Jnoncor_\red|$ corresponds to $\max_{\mu \in \Params_\textnormal{val}} |\Jhat_h(\mu) - \Jnoncor_\red(\mu)|$, and so forth.
	}
	\label{fig:estimator_study}
\end{figure}
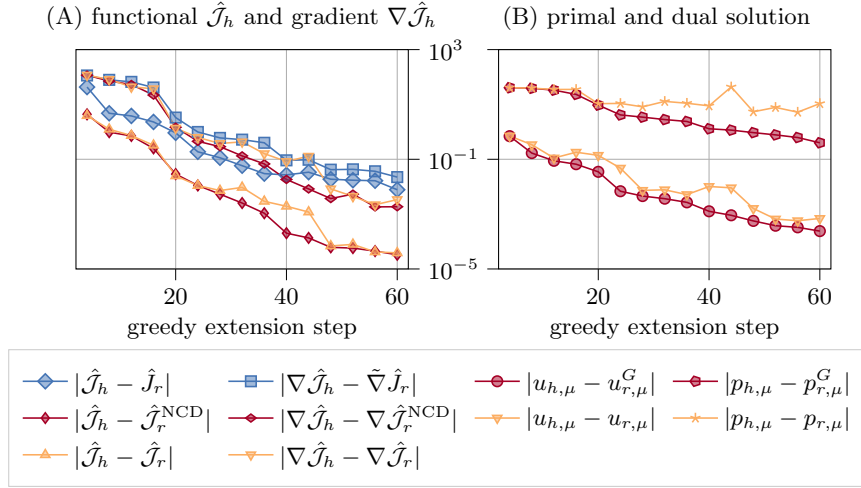

While producing the shown error study, we experienced instabilities of the PG reduced primal and dual systems.
This is due to the fact that for very complex dual problems, the test space of each problem very poorly fits to the respective ansatz space.
In fact, the decay of the primal error that can be seen in Figure \ref{fig:estimator_study}(B) can not be expected in general and is indeed a consequence of the simplification of the objective functional, where we chose $D \equiv \Omega$.
The stability problems can already be seen in the dual error and partly on the primal error for basis size 20.
In general it can happen that the reduced systems are highly unstable for specific parameter values.
Instead, a sophisticated Greedy Algorithm for deducing appropriate reduced spaces may require to build a larger dual or primal space by adding stabilizing snapshots, i.e. search for supremizers \cite{ballarin2015supremizer}.
Since the aim of the work at hand is not to provide an appropriate Greedy based algorithm, we instead decided to reduce the complexity of the functional where stability issues are less present.

\subsection{TR-RB algorithm}

We now compare the PG variant of the TR-RB with the above mentioned NCD-corrected TR-RB approach.
Importantly, we use the original version of the model problem, i.e. we pick the domain of interest to be defined as suggested in Figure \ref{ex1:blueprint}.
In fact, we accept the possibility of high instabilities in the reduced model.
We pick ten random starting parameter, perform both algorithms and compare the averaged result.
In Figure \ref{Fig:EXC12} one particular starting parameter is depicted and
in Table \ref{Tab:EXC12} the averaged results for all ten optimization runs are shown.
\begin{figure}
	\footnotesize
	\centering
	\begin{tikzpicture}
	\definecolor{color0}{rgb}{0.65,0,0.15}
	\definecolor{color1}{rgb}{0.84,0.19,0.15}
	\definecolor{color2}{rgb}{0.96,0.43,0.26}
	\definecolor{color3}{rgb}{0.99,0.68,0.38}
	\definecolor{color4}{rgb}{1,0.88,0.56}
	\definecolor{color5}{rgb}{0.67,0.85,0.91}
	\begin{axis}[
	name=left,
	width=6.5cm,
	height=4.5cm,
	log basis y={10},
	tick align=outside,
	tick pos=left,
	legend style={nodes={scale=0.7}, fill opacity=0.8, draw opacity=1, text opacity=1,
				  xshift=-1.6cm, yshift=-0.5cm, xshift=4.5cm, draw=white!80!black},
	x grid style={white!69.0196078431373!black},
	xlabel={time in seconds [s]},
	xmajorgrids,
	xtick style={color=black},
	y grid style={white!69.0196078431373!black},
	ymajorgrids,
	ymode=log,
	ylabel={\(\displaystyle \| \overline{\mu}-\mu^{(k)} \|^\text{rel}_2\)},
	ytick style={color=black}
	]
	\addplot [semithick, color0, mark=triangle*, mark size=3, mark options={solid, fill opacity=0.5}]
	table {%
        13.0756642818451 0.574286171259373
        31.8166162967682 0.574017027495371
        38.5897059440613 0.555477788898345
        45.570326089859 0.52429437114602
        52.9812984466553 0.507084209408289
        60.9106016159058 0.41660095554096
        69.3280558586121 0.375510754302739
        78.2450692653656 0.299488619035033
        88.5779669284821 0.113477915062374
        107.952816724777 0.0114236653909785
        135.231210231781 0.000147935221399132
        151.076913118362 2.58254983619181e-05
        164.675134181976 5.26070844357256e-06
	};
	\addlegendentry{BFGS NCD TR-RB \cite{keil2020nonconforming}}
	\addplot [semithick, color3, mark=*, mark size=3, mark options={solid, fill opacity=0.5}]
	table {%
        12.1607882976532 0.574286171259373
        35.4619424343109 0.574040233206862
        43.4497337341309 0.569126018670776
        52.2917582988739 0.56124503305629
        62.6727526187897 0.548418201394747
        74.8126885890961 0.525321267398303
        88.072719335556 0.524967245607288
        102.938640356064 0.519210519549374
        120.043780326843 0.518929447625113
        138.781098842621 0.504681744974713
        159.799144983292 0.478880041078375
        182.603423595428 0.428336507745106
        208.776380777359 0.405688185871396
        237.186156749725 0.30213730835205
        268.500371694565 0.236613269296937
        302.714802026749 0.202947643541194
        342.248579025269 0.0351113761140485
        384.453255176544 0.0291327949814106
        437.500003099442 3.37918525465338e-05
        486.676156520844 9.02321967828624e-07
	};
	\addlegendentry{BFGS PG TR-RB}
	\end{axis}
	\end{tikzpicture}
	\captionsetup{width=\textwidth}
	\caption{\footnotesize{%
			Relative error decay w.r.t. the optimal parameter $\bar\mu$ and performance of selected algorithms for a single optimization run with random initial guess $\mu^{(0)}$ for $\tau_\text{FOC} = 10^{-6}$.}}
	\label{Fig:EXC12}
\end{figure}
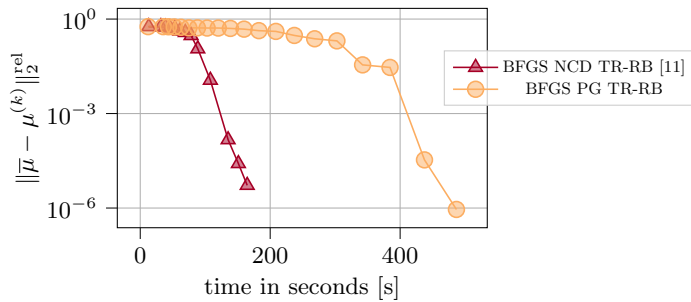

\begin{table}
	\footnotesize
	\centering
	\begin{tabular}{l|cc|c|cc}
		& runtime[s]
		& & iterations $k$\\
		& avg.~(min/max)
		& speed-up
		& avg.~(min/max)
		& rel.~error
		& FOC cond.\\\hline 
		FOM BFGS    & 6955~(4375/15556) & --&471.44~(349/799) & 3.98e-5 & 3.36e-6\\
		TR NCD-BFGS & 171~(135/215)& 40.72&~12.56(11/15) & 4.56e-6& 6.05e-7\\
		TR PG-BFGS  & 424~(183/609)& 16.39&~17.56(11/22) & 4.62e-6& 7.57e-7
	\end{tabular}
	\captionsetup{width=\textwidth}
	\caption{\footnotesize{%
			Performance and accuracy of the algorithms for ten optimization runs with random initial guess $\mu^{(0)}$ and $\tau_{\text{\rm{FOC}}}=10^{-6}$.
			\label{Tab:EXC12}}}
\end{table}

It can be seen that the PG variant is a valid approach and converges sufficiently fast with respect to the FOM BFGS method. 
Surely, it can not be said that this stability issues do not enter the performance of the proposed TR-RB methodology but regardless of the stability of the reduced system, we note that the convergence result in Theorem 1 still holds true. 
However, the instability of the reduced systems clearly harms the algorithm from iterating as fast as the NCD-corrected approach.
One reason for that is that the trust region is much larger for the NCD-corrected approach, allowing the method to step faster.
We would like to emphasize that the depicted result in Figure \ref{Fig:EXC12} is neither an instance of the worst nor the best performance of the PG approach but rather an intermediate performance.
The comparison highly depends on the starting parameter, and the structure of the optimization problem.
As discussed above, the suggested PG approach can potentially benefit from more involved enrichment strategies that account for the mentioned stability issues.

Last but not least, it is important to mention that the above experiment showed weaknesses of the chosen projected BFGS approach as FOM method as well as for the TR-RB subproblems which has been extensively studied in \cite{banholzer2020adaptive}.
Instead, it is beneficial to choose higher order optimization methods, such as projected Newton type methods, where also the PG variant can be applied to neglect the more involved contributions from the NCD-corrected approach that are entering the Hessian approximation.
 
 \section{Concluding remarks} 
In this contribution we demonstrated, how adaptive enrichment based on rigorous a posteriori error control can be used within a Trust-Region-Reduced-Basis approach to speed up the solution of large scale PDE constrained optimization problems. Within this approach the reduced approximation spaces are tailored towards the solution of the overall optimization problem and thus circumvent the offline construction of a reduced model with good approximation properties with respect to the whole parameter regime. In particular, we compared a new Petrov-Galerkin approach with the NCD-Galerkin approach that has recently been introduced in \cite{banholzer2020adaptive}. Our results demonstrate the benefits of the Petrov-Galerkin approach with respect to the approximation of the objective functional and its derivatives as well as with respect to corresponding a posteriori error estimation. However, the results also show deficiencies with respect to possible stability issues. Thus, further improvements of the enrichment strategy are needed in order to guarantee uniform boundedness of the reduced inf-sup constants.  
 
%
%
{
	\bibliographystyle{abbrv}

}

\end{document}